\documentclass[11pt,english]{article}
\usepackage[T1]{fontenc}
\usepackage[utf8]{inputenc}
\usepackage{amsmath}
\usepackage{amsthm}
\usepackage{amssymb}
\usepackage{esint}

\makeatletter
\theoremstyle{plain}
\newtheorem{thm}{\protect\theoremname}
  \theoremstyle{plain}
  \newtheorem{prop}[thm]{\protect\propositionname}
  \theoremstyle{definition}
  \newtheorem{defn}[thm]{\protect\definitionname}
  \theoremstyle{remark}
  \newtheorem{rem}[thm]{\protect\remarkname}
  \theoremstyle{plain}
  \newtheorem{lem}[thm]{\protect\lemmaname}
  \theoremstyle{definition}
  \newtheorem{example}[thm]{\protect\examplename}

\@ifundefined{date}{}{\date{October 25th 2016}}
\usepackage{braket}

\usepackage[english]{babel}
\providecommand{\definitionname}{Definition}
  \providecommand{\propositionname}{Proposition}
  \providecommand{\remarkname}{Remark}
\providecommand{\theoremname}{Theorem}

\theoremstyle{plain}

\theoremstyle{remark}

\allowdisplaybreaks

\@addtoreset{equation}{section}

\def\sqr#1#2{{\vcenter{\vbox{\hrule height .#2pt \hbox{\vrule
 width .#2pt height#1pt \kern#1pt \vrule
width .#2pt} \hrule height .#2pt}}}}

\def\ds{\begin{displaystyle}}
\def\eds{\end{displaystyle}}

\def\<{\langle }
\def\>{\rangle }

\def\R{\mathbb R}

\DeclareMathAlphabet{\mathonebb}{U}{bbold}{m}{n}                           %

\title{Path dependent equations  driven by H\"older processes}

\usepackage{authblk}

\author{Rafael Andretto Castrequini\thanks{rafael.andretto-castrequini@ensta-paristech.fr}}
\author{Francesco Russo\thanks{francesco.russo@ensta-paristech.fr}}
\affil[1]{ENSTA ParisTech, Universit\'e Paris-Saclay,
 Unit\'e de Math\'ematiques appliqu\'ees, 828, Boulevard des Mar\'echaux, F-91120 Palaiseau, France}

\usepackage{babel}
\providecommand{\definitionname}{Definition}
  \providecommand{\propositionname}{Proposition}
  \providecommand{\remarkname}{Remark}
\providecommand{\theoremname}{Theorem}

  \providecommand{\definitionname}{Definition}
  \providecommand{\lemmaname}{Lemma}
  \providecommand{\propositionname}{Proposition}
  \providecommand{\remarkname}{Remark}
\providecommand{\theoremname}{Theorem}

\usepackage{babel}
\providecommand{\definitionname}{Definition}
  \providecommand{\lemmaname}{Lemma}
  \providecommand{\propositionname}{Proposition}
  \providecommand{\remarkname}{Remark}
\providecommand{\theoremname}{Theorem}

\usepackage{babel}
\providecommand{\definitionname}{Definition}
  \providecommand{\examplename}{Example}
  \providecommand{\lemmaname}{Lemma}
  \providecommand{\propositionname}{Proposition}
  \providecommand{\remarkname}{Remark}
\providecommand{\theoremname}{Theorem}

\usepackage{babel}
\providecommand{\definitionname}{Definition}
  \providecommand{\examplename}{Example}
  \providecommand{\lemmaname}{Lemma}
  \providecommand{\propositionname}{Proposition}
  \providecommand{\remarkname}{Remark}
\providecommand{\theoremname}{Theorem}

\usepackage{babel}
\providecommand{\definitionname}{Definition}
  \providecommand{\examplename}{Example}
  \providecommand{\lemmaname}{Lemma}
  \providecommand{\propositionname}{Proposition}
  \providecommand{\remarkname}{Remark}
\providecommand{\theoremname}{Theorem}

\usepackage{babel}
\providecommand{\definitionname}{Definition}
  \providecommand{\examplename}{Example}
  \providecommand{\lemmaname}{Lemma}
  \providecommand{\propositionname}{Proposition}
  \providecommand{\remarkname}{Remark}
\providecommand{\theoremname}{Theorem}

\usepackage{babel}
\providecommand{\definitionname}{Definition}
  \providecommand{\examplename}{Example}
  \providecommand{\lemmaname}{Lemma}
  \providecommand{\propositionname}{Proposition}
  \providecommand{\remarkname}{Remark}
\providecommand{\theoremname}{Theorem}

\makeatother

\usepackage{babel}
  \providecommand{\definitionname}{Definition}
  \providecommand{\examplename}{Example}
  \providecommand{\lemmaname}{Lemma}
  \providecommand{\propositionname}{Proposition}
  \providecommand{\remarkname}{Remark}
\providecommand{\theoremname}{Theorem}

\begin{document}
\allowdisplaybreaks \maketitle 
\begin{abstract}
This paper investigates existence results for path-dependent
 differential equations driven by a Hölder function
where the integrals are  understood in the Young sense.  
The two main results are proved via an application of
Schauder theorem and the vector field is allowed to be unbounded.
The Hölder function is typically the trajectory of
a stochastic process.
\end{abstract}


\section{Introduction}

The aim of the paper is to discuss existence theorems for a
path-dependent equation of the type 
\begin{equation} \label{EPrincipal}
Y_{t}=y_{0}+\int_{0}^{t}F\left(u,Y_{u}\left(\centerdot\right)\right)dX_{u}, \ t\in [0,T],
\end{equation}
where $X$ is an $\alpha$-Hölder continuous  $n$-dimensional 
process and $F$ is a path-dependent 
$m \times n$ matrix-valued vector field
defined on $\left[0,T\right]\times C(\left[0,T\right];{\mathbb R}^m)$.
The trajectories of the unknown process $Y$ are ${\mathbb R}^m$-valued 
$\alpha$-H\"older continuous functions;
 $Y_{u}\left(\centerdot\right)$
denotes the trajectory of $Y$ until time $u$, i.e, $Y_{u}\left(x\right)=Y_{x}$,
if $x\le u$ and $Y_{u}\left(x\right)=Y_{u}$ otherwise.
The integral is intended in the Young sense so that 
$F$ needs of course  to verify  a Hölder type regularity,
 see Sec. \ref{sec:Young-integral}
for details. 

Path-dependent (similarly to functional dependent or delay) equations have a
long story. 
 To the best of our knowledge  
 the first author who has contributed in this framework in the stochastic case, is \cite{cho}  
 motivated by \cite{delfour}, which is a significant contribution
 in the deterministic case. A relevant
 monograph in the subject is the one of \cite{mohammed}. Considerations
 about functional-dependent equations in law also appear in \cite{rw}.
 More recently several studies have been performed studying the relation
 between functional dependent equations and path-dependent PDEs see
 e.g. \cite{DGR, DGRnote, flandoli_zanco13, cosso_russo15a} in the framework of Banach space valued stochastic calculus and many others in the framework of 
path-dependent functional It\^o calculus, see \cite{leao_ohashi_simas14}
and references therein for latest developments. 

Young integral was introduced first in \cite{young1936inequality}.
A recent paper on the subject is \cite{gubinelli2004controlling}
and an excellent monograph recalling Young integral in the perspective
of rough paths is \cite{HairerBook}. That integral has been implemented
for the study of ordinary stochastic differential equations driven
by Hölder processes (see \cite{lyons-young, lejay:inria-00402397})
 and also SPDEs, see \cite{maslowski, gubinelli2006young}.
As far as we know, the present paper is the first one which discusses
functional-dependent equations in the framework of Young integral.

The aim of this paper is to discuss existence results 
 under suitable minimal assumptions on $F$. 
Contrarily to most of the literature on Young differential equations even in
the non-path dependent setting, the authors allow the
vector field  $F$ to be unbounded.
The main results about existence are Theorem \ref{thm:F_is_sublinear}
and Theorem \ref{T20};  the latter supposes $F$ to be bounded 
but with less restrictive H\"older type conditions
on $F$. The path-dependent framework however offers
other perspectives of generalization if one
assumes a different type on dependence
on the past trajectory.
For instance in Section \ref{sec:Particular-Case} we remark that, whenever the
dependence of $F$ with respect to the past allows a gap with
respect to the present, the construction of a solution can be
done iteratively.


\section{Preliminaries}

In this section we introduce some basic definitions.

 Let $U$ and $V$ be  Banach spaces and denote by $L=L\left(V,U\right)$
 the space of continuous linear maps from $V$ to $U$.

 We reserve the symbols $X$ to denote driving paths of our differential
 equation. Typically $X:\left[0,T\right]\rightarrow V$ is an $\alpha$-Hölder
 continuous. Hence there is a constant (the smallest one is denoted
 by $\left\Vert X\right\Vert _{\alpha}$) such that 
\[
\left|X_{t}-X_{s}\right|\le\left\Vert X\right\Vert _{\alpha}\left|t-s\right|^{\alpha},
\]
 for all $s,t\in\left[0,T\right]$. As usual we write, $X\in C^{\alpha}\left(\left[0,T\right],V\right)$
 and  
$$\left\Vert X\right\Vert _{\alpha}=\sup_{s,t\in\left[0,T\right], s \neq t
}\frac{\left|X_{t}-X_{s}\right|}{\left|t-s\right|^{\alpha}}.$$

It is sometimes useful to specify the closed interval $I=[t_{0},t_{1}]$
 where we evaluate $\left\Vert X\right\Vert _{\alpha}$. For this matter,
 we further define 
\[
\left\Vert X\right\Vert _{\alpha;I}:=\sup_{s,t\in I, s \neq t}\frac{\left|X_{t}-X_{s}\right|}{\left|t-s\right|^{\alpha}}.
\]
If $t_{0}\ge0$ and $\tau>0$ we will simply denote 
\[
\left\Vert X\right\Vert _{\alpha;t_{0},\tau}:=\left\Vert X\right\Vert _{\alpha;\left[t_{0},t_{0}+\tau\right]}.
\]

Usually we omit the symbol $U$ in $C^{\alpha}\left(I,U\right)$
becoming simply $C^{\alpha}\left(I\right)$. 
It is well-known that $\left\Vert \centerdot\right\Vert _{\alpha;I}$
induces semi-distance on the vector space $C^{\alpha}(I)$, though
we can endow it with a metric by setting $\left|\left|\left|Z-Y\right|\right|\right|_{\alpha;I}:=\left\lVert Z-Y\right\rVert _{\alpha;I}+\left|Z_{t_{0}}-Y_{t_{0}}\right|$,
for all $Z,Y\in C^{\alpha}\left(I\right)$. An important closed
subset of $C^{\alpha}\left(I\right)$ is obtained by fixing an \emph{initial
condition}. That is, for a fixed $a\in U$, we define $C_{a}^{\alpha}\left(I\right):=\left\{ Y\in C^{\alpha}\left(I\right)\ |\ Y_{t_{0}}=a\right\} $
and this is a closed subset of $C^{\alpha}\left(I\right)$. Moreover,
the metric $\left|\left|\left|Z-Y\right|\right|\right|_{\alpha;I}$
on $C_{a}^{\alpha}\left(I\right)$ coincides with $\left\Vert \centerdot\right\Vert _{\alpha;I}$,
i.e. $\left|\left|\left|Z-Y\right|\right|\right|_{\alpha;I}=\left\lVert Z-Y\right\rVert _{\alpha;I}$,
for all $Z,Y\in C^{\alpha}\left(I\right)$. For the sake of
clarity, when we refer to $C^{\alpha}$-topology we mean the topology
induced by $\left|\left|\left|\centerdot\right|\right|\right|_{\alpha;I}$, 
which coincides with $\left\lVert \centerdot\right\rVert _{\alpha;I}$
in the subset $C_{a}^{\alpha}\left(I\right)$. The set $C\left(I\right)$
is the usual Banach space of continuous functions equipped with the
$\sup$-norm, $\left\lVert Y\right\rVert _{\infty;I}:=\sup_{t\in I}\left|Y_{t}\right|$.

Since our differential equation involves a vector field acting on the whole
trajectory of a path, we make use of the following notation. Given
a continuous path $Y:\left[0,T\right]\rightarrow U$ we denote by
the calligraphic version of $Y$, namely $\mathcal{Y}$, 
onto $C\left(\left[0,T\right],U\right)$, i.e. 
\[
\mathcal{Y}:\underset{t}{\left[0,T\right]}\underset{\mapsto}{\rightarrow}\underset{\mathcal{Y}_{t}}{C\left(\left[0,T\right],U\right)},
\]
with 
\begin{equation}
\mathcal{Y}_{t}\left(x\right):=Y_{t\wedge x}.\label{Emin}
\end{equation}

 We observe that, if $Y$ is
$\alpha$-Hölder continuous then its lift is $\alpha$-Hölder as well.
Indeed, this is shown in the statement below.

\begin{prop}
Let $Y\in C^{\alpha}([0,T])$. Then, for any $0\le a\le b\le T$,
$\mathcal{Y}\in C^{\alpha}\left(\left[a,b\right],C\left[0,T\right]\right)$
we have
\begin{align}
\left\Vert \mathcal{Y}\right\Vert _{\alpha;\left[a,b\right]} & \le\left\Vert Y\right\Vert _{\alpha;\left[0,T\right]}\label{eq:Y_lift_holdernorm}\\
\left\lVert \mathcal{Y}_{t}\right\rVert _{\infty;\left[0,T\right]} & \le\left\Vert Y\right\Vert _{\alpha;\left[0,T\right]}\label{eq:Y_lift_holdernorm1}.
\end{align}
\end{prop}
\begin{proof}
Fix $s\le t$ in $\left[a,b\right]$. For any $x\in\left[0,T\right]$
we have 
\begin{align*}
\left|\mathcal{Y}_{t}\left(x\right)-\mathcal{Y}_{s}\left(x\right)\right| & \le\begin{cases}
\left|Y_{x}-Y_{x}\right| & ,\,x\in\left[0,s\right]\\
\left|Y_{x}-Y_{s}\right| & ,\,x\in\left[s,t\right]\\
\left|Y_{t}-Y_{s}\right| & ,\,x\in\left[t,T\right]
\end{cases}\\
 & \le\begin{cases}
0 & ,\,x\in\left[0,s\right]\\
\left\Vert Y\right\Vert _{\alpha}\left|x-s\right|^{\alpha} & ,\,x\in\left[s,t\right]\\
\left\Vert Y\right\Vert _{\alpha}\left|t-s\right|^{\alpha} & ,\,x\in\left[t,T\right]
\end{cases}\\
 & \le\left\Vert Y\right\Vert _{\alpha}\left|t-s\right|^{\alpha},
\end{align*}
so $\left\lVert \mathcal{Y}_{t}-\mathcal{Y}_{s}\right\rVert _{\infty;\left[0,T\right]}\le\left\Vert Y\right\Vert _{\alpha}\left|t-s\right|^{\alpha}$.
This proves $\mathcal{Y}\in C^{\alpha}\left(\left[a,b\right],C([0,T])\right)$
with $\left\Vert \mathcal{Y}\right\Vert _{\alpha;\left[a,b\right]}\le\left\Vert Y\right\Vert _{\alpha;\left[0,T\right]}$
i.e. equality \eqref{eq:Y_lift_holdernorm}.


Now, fix $t\in\left[a,b\right]$ and let any $x,y\in\left[0,T\right]$.
By definition, see   \eqref{Emin}, it follows 
\[
\left|\mathcal{Y}_{t}\left(x\right)-\mathcal{Y}_{t}\left(y\right)\right|\le\left\Vert Y\right\Vert _{\alpha}\left|x-y\right|^{\alpha},
\]
hence $\mathcal{Y}_{t}\in C^{\alpha}([0,T])$ and equality \eqref{eq:Y_lift_holdernorm1}
holds.
\end{proof}

\section{\label{sec:Young-integral}Young integral}

\label{S3}

At this level we recall the fundamental inequality characterizing
Young integral. For a complete treatment we refer the reader \cite[Ch. 6]{friz2010multidimensional}.
For a reference closer to our spirit we recommend \cite[pp.47-48;63]{HairerBook}.
We remark that the second inequality below is a consequence of the
first one.

\begin{thm}
\label{thm:Young_integral} Let $t_{0}\in\left[0,T\right]$, $a\in U$
and let $\tau>0$ such that $\tau+t_{0}\le T$. Given $W\in C^{\gamma}\left(\left[t_{0},t_{0}+\tau\right],L\left(V,U\right)\right)$
and $X\in C^{\alpha}\left(\left[0,T\right],V\right)$ with $\alpha+\gamma>1$, the map
\begin{align}\label{eq:yi_holder}
W\mapsto I\left(W\right):=a+\int_{t_{0}}^{\centerdot}W_{u}dX_{u}\in C^{\alpha}([0,T]),
\end{align}
is continuous and it satisfies, for $s\le t$ in $\left[t_{0},t_{0}+\tau\right]$, 
\begin{align}\label{eq:Young_inequality_general}
\left|\int_{s}^{t}W_{u}dX_{u}-W_{s}\left(X_{t}-X_{s}\right)\right| & \le k_{\alpha+\gamma}\left\lVert W\right\rVert _{\gamma;\left[s,t\right]}\left\lVert X\right\rVert _{\alpha;\left[s,t\right]}\left|t-s\right|^{\alpha+\gamma},
\end{align}
where $k_{\mu}=\frac{1}{1-2^{1-\mu }}$.

If furthermore $\tau\le1$, we can write the inequality above as 
\begin{equation}\label{eq:Young_inequality}
\left\Vert I\left(W\right)\right\Vert _{\alpha;t_{0},\tau}\le\left(k_{\alpha+\gamma}+1\right)\left\Vert X\right\Vert _{\alpha;t_{0},\tau}\left(\left\Vert W\right\Vert _{\gamma;t_{0},\tau}+\left|W_{t_{0}}\right|\right).
\end{equation}

\end{thm}

\section{The vector field $F$}

In this section we formally introduce the driving vector field of
the equation. For every $t \in [0,T]$ and $Y \in C([0,T],U)$,
$F(t,Y)$ is a linear map acting on  $V$ to  $U$.
Also we will present some fundamental inequalities regarding 
composition maps, such as $Y\mapsto F\left(t,\mathcal{Y}_{t}\right)$.

\begin{defn}
\label{def:F_is_non_ant}Let 
\[
F:\left[0,T\right]\times C\left(\left[0,T\right],U\right)\rightarrow L.
\]
We will say that it is \textbf{non-anticipating} if it satisfies 
\[
F\left(t,Y\right)=F\left(t,\mathcal{Y}_{t}\right),
\]
for all $Y\in C\left[0,T\right]$ and $t\in\left[0,T\right]$ \end{defn}
The requirement above means that $F\left(t,Y\right)$ does not depend
on what happened on $\left.Y\right|_{\left[t,T\right]}$. It will
indeed fulfill the property below. 

\begin{rem}
Given $Y$,$\tilde{Y}\in C\left[0,T\right]$ such that $Y|_{\left[0,s\right]}=\tilde{Y}|_{\left[0,s\right]}$.
Then 
\[
F\left(s,Y\right)=F\left(s,\tilde{Y}\right).
\]

\end{rem}
We assume Hölder regularity on $F$ as follows. The forthcoming examples
will play as motivation for the definition below. 
\begin{defn}
\label{def:F_spcial_holder} The vector field $F:\left[0,T\right]\times C^{\alpha}\left(\left[0,T\right],U\right)\rightarrow L$
will be said 
\textbf{ $\boldsymbol{\left(\alpha,\beta\right)}$-Hölder} if 
there are some non-negative constants $c_{\alpha,\beta}$ and $\tilde{c}_{\alpha,\beta}$,
\begin{align}
\left|F\left(t,Y\right)-F\left(s,Y\right)\right| & \le c_{\alpha,\beta}\left(1+\left\lVert Y\right\rVert _{\alpha,\left[s,t\right]}^{\beta}\right)\left|t-s\right|^{\alpha\beta},\ \forall s,t\in[0,T],\label{eq:F_is_spec_holder}\\
\left|F\left(s,Y\right)-F\left(s,\tilde{Y}\right)\right| & \le\tilde{c}_{\alpha,\beta}\left(\left\lVert Y-\tilde{Y}\right\rVert _{\alpha,\left[0,s\right]}+\left|Y_{0}-\tilde{Y}_{0}\right|\right)^{\beta},\ \forall s\in[0,T], \label{eq:F_is_spe_lip}
\end{align}
for all $Y, \tilde Y \in C^{\alpha}\left(\left[0,T\right],U\right).$
\end{defn}
\begin{rem}
\label{RSpecial} 
\begin{enumerate}
\item If $F$ is 
 $\left(\alpha,\beta\right)$-Hölder 
then  (\ref{eq:F_is_spe_lip}) implies that $F$ is non-anticipating.
\item It is well-known that, for $\alpha'<\alpha$, if $Y$ is $\alpha$-Hölder
then $Y$ is $\alpha'$-Hölder. This follows from the inequality $\left\lVert Y\right\rVert _{\alpha';\left[t_{0},t_{1}\right]}\le\left\lVert Y\right\rVert _{\alpha;\left[t_{0},t_{1}\right]}\left|t_{1}-t_{0}\right|^{\alpha-\alpha'}$.
This remark motivates the  lemma below, which will be useful in
the next sections. 
\end{enumerate}
\end{rem}
\begin{lem}
\label{lem:F_is_alpha_prime_holder} 
Let $0<\alpha'\le\alpha\le1$
and $\beta\in (0,1].$ Suppose that $F:\left[0,T\right]\times C^{\alpha'}\left(\left[0,T\right],U\right)\rightarrow L$
is an $\left(\alpha',\beta\right)$-Hölder continuous vector field.
Then $F$ is $\left(\alpha,\beta\right)$-Hölder continuous with constants
$c_{\alpha,\beta}:=c_{\alpha',\beta}\left(T^{\alpha'\beta}\vee1\right)$
and $\tilde{c}_{\alpha,\beta}:=\tilde{c}_{\alpha',\beta}\left(T^{(\alpha-\alpha')\beta}\vee1\right)$.\end{lem}
\begin{proof}
Let $s<t$ belonging to $\left[0,T\right]$ and let 
$Y\in C^{\alpha}([0,T])$.
Using the inequality $\left\lVert Y\right\rVert _{\alpha';\left[s,t\right]}\le\left\lVert Y\right\rVert _{\alpha;\left[s,t\right]}\left|t-s\right|^{\alpha-\alpha'}$
and $\left(\alpha',\beta\right)$-Hölder continuity of $F$ (below
we are going to omit the sub-index $\beta$ related to the constants $c_{\alpha,\beta}$,
$\tilde{c}_{\alpha,\beta}$ , $c_{\alpha',\beta}$ and $\tilde{c}_{\alpha',\beta}$),
we obtain
\begin{align*}
\left|F\left(t,Y\right)-F\left(s,Y\right)\right| & \le c_{\alpha'}\left(1+\left\lVert Y\right\rVert _{\alpha';\left[s,t\right]}^{\beta}\right)\left|t-s\right|^{\alpha'\beta}\\
 & \le c_{\alpha'}\left(1+\left\lVert Y\right\rVert _{\alpha;\left[s,t\right]}^{\beta}\left|t-s\right|^{\left(\alpha-\alpha'\right)\beta}\right)\left|t-s\right|^{\alpha'\beta}\\
 & =c_{\alpha'}\left(\left|t-s\right|^{\alpha'\beta}+\left\lVert Y\right\rVert _{\alpha;\left[s,t\right]}^{\beta}\left|t-s\right|^{\alpha\beta}\right)\\
 & \le c_{\alpha'}\left(T^{\alpha'\beta}+\left\lVert Y\right\rVert _{\alpha;\left[s,t\right]}^{\beta}\left|t-s\right|^{\alpha\beta}\right)\\
 & \le c_{\alpha'}\left(T^{\alpha'\beta}\vee1\right)\left(1+\left\lVert Y\right\rVert _{\alpha;\left[s,t\right]}^{\beta}\right)\left|t-s\right|^{\alpha'\beta},
\end{align*}
hence  equality (\ref{eq:F_is_spec_holder}) holds with $c_{\alpha}=c_{\alpha'}\left(T^{\alpha'\beta}\vee1\right)$.
 It remains to show \eqref{eq:F_is_spe_lip}. 
 Let $Y,\tilde{Y}\in C^{\alpha}([0,T])$
and $s\in\left[0,T\right]$. We have
\begin{align*}
\left|F\left(s,Y\right)-F\left(s,\tilde{Y}\right)\right| & \le\tilde{c}_{\alpha'}\left(\left\lVert Y-\tilde{Y}\right\rVert _{\alpha',\left[0,s\right]}+\left|Y_{0}-\tilde{Y}_{0}\right|\right)^{\beta}\\
 & \le\tilde{c}_{\alpha'}\left(\left\lVert Y-\tilde{Y}\right\rVert _{\alpha,\left[0,s\right]}s^{\alpha-\alpha'}+\left|Y_{0}-\tilde{Y}_{0}\right|\right)^{\beta}\\
 & \le\tilde{c}_{\alpha'}\left(T^{(\alpha-\alpha')\beta}\vee1\right)\left(\left\lVert Y-\tilde{Y}\right\rVert _{\alpha,\left[0,s\right]}+\left|Y_{0}-\tilde{Y}_{0}\right|\right),
\end{align*}
which proves the claim with 
$\tilde{c}_{\alpha}=\tilde{c}_{\alpha'}\left(T^{(\alpha-\alpha')\beta}\vee1\right)$.
\end{proof}
\begin{example} 
B. Dupire (see \cite{dupire2009functional}) introduced a notion of
non-anticipating functional in the sense introduced below. Denote $\Lambda:=\cup_{s\in\left[0,T\right]}\Lambda^{s}$ and $\Lambda^{s}:=C([0,s])$. We endow $\Lambda$    with the metric $d$ defined as follows. For any $Z^{t},Y^{s}\in\Lambda$ (the superscript here means  that $Z^{t}\in\Lambda^{t}$ and $Y^{s}\in\Lambda^{s}$) assuming $s\le t$, he defines
 $d\left(Z^{t},Y^{s}\right):=\left|t-s\right|+\left\lVert Z^{t}-\mathcal{Y}_{s}^{s}\right\rVert _{\infty;[0,t]}$,
where similarly as at \eqref{Emin}, we set 
$\mathcal{Y}^s_{u}\left(x\right):=Y^s_{u\wedge x}, u, x \in [0,t].$
Let $f:\Lambda\rightarrow\mathbb{R}$ be 
 a $\beta$-Hölder continuous functional with respect to $d$.
 Then the vector field $F$ defined by
\[ F\left(t,Y\right):=f\left(\left.Y\right|_{\left[0,t\right]}\right),\   
\forall t\in [0,T],\  \forall Y\in C^\alpha ([0,T]),
\] 
is non-anticipating in the sense of Definition \ref{def:F_is_non_ant}
and it is $\left(\alpha,\beta\right)$-Hölder.

Indeed, for $Y\in C^{\alpha}([0,T])$, we have 
\begin{align*}
\left|F\left(t,Y\right)-F\left(s,Y\right)\right| & =\left|f\left(\left.Y\right|_{\left[0,t\right]}\right)-f\left(\left.Y\right|_{\left[0,s\right]}\right)\right|\\
 & \le\left\Vert f\right\Vert _{\beta}d\left(\left.Y\right|_{\left[0,t\right]},\left.Y\right|_{\left[0,s\right]}\right)^{\beta}\\
 & \le\left\Vert f\right\Vert _{\alpha}\left(\left|t-s\right|+\left\lVert Y-\mathcal{Y}_{s}\right\rVert _{\infty;0,t}\right)^{\beta}\\
 & =\left\Vert f\right\Vert _{\alpha}\left(\left|t-s\right|+\sup_{x\in\left[s,t\right]}\left|Y_{x}-Y_{s}\right|\right)^{\beta}\\
 & \le\left\Vert f\right\Vert _{\alpha}\left(\left|t-s\right|+\left\Vert Y\right\Vert _{\alpha;\left[s,t\right]}\left|t-s\right|^{\alpha}\right)^{\beta}\\
 & \le\left\Vert f\right\Vert _{\alpha}\left(T^{1-\alpha}+\left\Vert Y\right\Vert _{\alpha;\left[s,t\right]}\right)^{\beta}\left|t-s\right|^{\alpha\beta}.
\end{align*}
This proves (\ref{eq:F_is_spec_holder}).

It remains to  prove  (\ref{eq:F_is_spe_lip}). Given $Y,Z\in 
C^{\alpha}([0,T])$ and
 $s \in [0,T]$, 
\begin{align*}
\left|F\left(s,Y\right)-F\left(s,Z\right)\right| & =\left|f\left(\left.Y\right|_{\left[0,s\right]}\right)-f\left(\left.Z\right|_{\left[0,s\right]}\right)\right|\\
 & \le\left\lVert f\right\rVert _{\beta}\left\{ \left|s-s\right|+\left\lVert Y-Z\right\rVert _{\infty;\left[0,s\right]}\right\} ^{\beta}\\
 & \le\left\lVert f\right\rVert _{\beta}\left\{ 0+\left\lVert Y-Z\right\rVert _{\alpha;\left[0,s\right]}s^{\alpha}+\left|Y_{0}-Z_{0}\right|\right\} ^{\beta}.
\end{align*}
\end{example}

\begin{example}
\emph{Young integral functional}. Let $\alpha,\gamma\in(0,1]$ with 
 $\alpha+\gamma>1$. Fix a function $g$ in $C^\gamma ([0,T])$.
Define the vector field $F_{g}$ by setting $F_{g}\left(t,Y\right):=\int_{0}^{t}g_{u}dY_{u}$,
for each $Y\in C^{\alpha}([0,T])$. Then $F_{g}$ is  $\left(\alpha,1\right)$-Hölder continuous. Indeed, using
Young integral inequality (\ref{eq:Young_inequality_general}) with $W$ (and $X$) instead of  $g$ (and $Y$), for any $s,t\in\left[0,T\right]$
and $Y \in C^\alpha ([0,T])$, it follows that
\begin{align*}
\left|F_{g}\left(t,Y\right)-F_{g}\left(s,Y\right)\right| & = \left| \int_{s}^{t}g_{u}dY_{u} \right| \\
 & \le k_{\alpha+\gamma}\left\lVert g\right\rVert _{\gamma;\left[s,t\right]}\left\lVert Y\right\rVert _{\alpha;\left[s,t\right]}\left|t-s\right|^{\alpha+\gamma}+\left|g_{s}\right|\left|Y_{t}-Y_{s}\right|\\
 & \le\left\{ k_{\alpha+\gamma}\left\lVert g\right\rVert _{\gamma;\left[0,T\right]}T^{\gamma}+\left\lVert g\right\rVert _{\infty;\left[0,T\right]}\right\} \left\lVert Y\right\rVert _{\alpha;\left[s,t\right]}\left|t-s\right|^{\alpha},
\end{align*}
which proves (\ref{eq:F_is_spec_holder}). 

Now, given any $Y,Z\in C^{\alpha }([0,T])$ and $s\in [0,T]$, we apply previous
inequality with $Y-Z$ instead of $Y$ and for $s= 0$. Thus, we obtain
\begin{align*}
\left|F_{g}\left(s,Y\right)-F_{g}\left(s,Z\right)\right| & =\left| \int_{0}^{s}g_{u}d(Y-Z)_{u} \right| \\
 & \le\left\{ k_{\alpha+\gamma}\left\lVert g\right\rVert _{\gamma;\left[0,T\right]}T^{\gamma}+\left\lVert g\right\rVert _{\infty;\left[0,T\right]}\right\} \left\lVert Y-Z\right\rVert _{\alpha;\left[0,s\right]}T^{\alpha},
\end{align*}
which proves \eqref{eq:F_is_spe_lip}. 
\end{example}

\begin{example}
\emph{Young Integral functional (continued). }More generally, we
 consider the vector field $F\left(t,Y\right):=h\left(t,\int_{0}^{t}g_{u}^{1}dY_{u},\ldots,\int_{0}^{t}g_{u}^{N}dY_{u}\right)$,
where $g^{i}\in C^{\gamma_{i}}([0,T])$ with $\alpha+\gamma_{i}>1, 
1 \le i \le N,$
and $h:\left[0,T\right]\times\mathbb{R}^{N}\rightarrow\mathbb{R}$
such there is a constant $K>0$ with
\[
\left|h\left(t,b\right)-h\left(s,a\right)\right|\le K\left\{ \left|t-s\right|^{\alpha\beta}+\max_{i=1,\ldots,N}\left|b_{i}-a_{i}\right|^{\beta}\right\},\ \forall t\in [0,T], \forall a,b\in \mathbb{R}^{N}.
\]
Then $F$ is $\left(\alpha,\beta\right)$-Hölder. Indeed, it is easy
to see that $F$ satisfies the inequalities 
\begin{align*}
\left|F\left(t,Y\right)-F\left(s,Y\right)\right| & \le K\left(1+\max_{i=1,\ldots,N}c_{i}\left\lVert Y\right\rVert _{\alpha;\left[s,t\right]}^{\beta}\right)\left|t-s\right|^{\alpha\beta},\\
\left|F\left(s,Y\right)-F\left(s,Z\right)\right| & =\left|F\left(s,Y-Z\right)\right|\le K\max_{i=1,\ldots,N}c_{i}\left\lVert Y-Z\right\rVert _{\alpha;\left[0,s\right]}^{\beta}T^{\alpha\beta},
\end{align*}
for any  $Y,Z\in C^{\alpha}$ and $s,t\in\left[0,T\right]$ where
$c_{i}:=\left(k_{\alpha+\gamma_{i}}\left\lVert g_{i}\right\rVert _{\gamma_{i}}T^{\gamma_{i}}+\left\lVert g_{i}\right\rVert _{\infty}\right)^{\beta}$.
\end{example} 

\begin{rem}
\label{rem:Geometric_interpolation} In the proposition below, we
will use a simple technique involving Hölder norms inequalities. 
It is called \textbf{geometric interpolation}
(in contrast to the linear one, $a\le\theta a+\left(1-\theta\right)b\le b$)
which states that,
whenever $W\in C^{\alpha}$ and $\theta\in\left(0,1\right)$, then
\begin{equation} \label{geo_int1}
\left\lVert W\right\rVert _{\alpha\theta}\le\left\lVert W\right\rVert _{0}^{1-\theta}\left\lVert W\right\rVert _{\alpha}^{\theta},
\end{equation}
recalling the notation 
$\left\Vert W\right\Vert _{0}=\sup_{s,t}\left|W_{t}-W_{s}\right|$.
The proof of \eqref{geo_int1} is a consequence of the equality $\frac{\left|W_{t}-W_{s}\right|}{\left|t-s\right|^{\alpha\theta}}=\left|W_{t}-W_{s}\right|^{1-\theta}\left(\frac{\left|W_{t}-W_{s}\right|}{\left|t-s\right|^{\alpha}}\right)^{\theta}$.
In particular we get 
\begin{equation}\label{geo_int2}
\left\lVert W\right\rVert _{\alpha\theta}\le2^{1-\theta}\left\lVert W\right\rVert _{\infty}^{1-\theta}\left\lVert W\right\rVert _{\alpha}^{\theta}.
\end{equation}
\end{rem}

\begin{prop}\label{prop:FoY_is_continuous}
	Let $F:\left[0,T\right]\times C^{\alpha}\left(\left[0,T\right],U\right)\rightarrow L$
be a non-anticipating and $\left(\alpha,\beta\right)$-Hölder continuous
vector field.

Fix $t_{0},t_{1}\in\left[0,T\right]$. Given $Y,Z\in C^{\alpha}\left(\left[0,t_{1}\right],U\right)$ 
 the inequality 
\begin{equation}
\left\lVert F\left(\centerdot,Y\right)\right\rVert _{\alpha\beta;\left[t_{0},t_{1}\right]}\le c_{\alpha,\beta}\left(1+\left\lVert Y\right\rVert _{\alpha;\left[t_{0},t_{1}\right]}^{\beta}\right)\label{eq:Comp1}
\end{equation}
holds. Moreover, if $\left\Vert Y\right\Vert _{\alpha;\left[t_{0},t_{1}\right]},\left\Vert Z\right\Vert _{\alpha;\left[t_{0},t_{1}\right]}\le R$, 
for any $\theta\in\left(0,1\right)$, it follows
\begin{equation}
\left\lVert F\left(\centerdot,Y\right)-F\left(\centerdot,Z\right)\right\rVert _{\alpha\beta\theta;\left[t_{0},t_{1}\right]}\le2\tilde{c}_{\alpha,\beta}^{1-\theta}c_{\alpha;\beta}^{\theta}\left(1+R^{\beta}\right)^{\theta}\left(\left\lVert Y-Z\right\rVert _{\alpha;0,t_{1}}+\left|Y_{0}-\tilde{Y}_{0}\right|\right)^{\beta\left(1-\theta\right)},\label{eq:Comp2}
\end{equation}
where $c_{\alpha,\beta}$ and $\tilde{c}_{\alpha,\beta}$
are the constants introduced in Definition \ref{def:F_spcial_holder}. \end{prop}
\begin{proof}
For $s<t$ in $\left[t_{0},t_{1}\right]$, and $Y\in C^{\alpha}([0,t_{1}])$,
 (\ref{eq:Comp1}) follows directly from the definition of $\left(\alpha,\beta\right)$-Hölder continuous.

We prove now (\ref{eq:Comp2}). We set $I:=\left[t_{0},t_{1}\right]$.
Given $Y,Z\in C^{\alpha}([0,t_{1}])$ with $\left\lVert Y\right\rVert _{\alpha;I},\left\lVert Z\right\rVert _{\alpha;I}\le R$,
we write $W_{t}:=F\left(t,Y\right)-F\left(t,Z\right)$, $t\in I$.
Fix an arbitrary $\theta\in\left(0,1\right)$; using \eqref{geo_int2}
we obtain
\begin{equation}\label{w1}
\left\lVert W\right\rVert _{\alpha\beta\theta;I}\le2^{1-\theta}\left\lVert W\right\rVert _{\infty;I}^{1-\theta}\left\lVert W\right\rVert _{\alpha\beta;I}^{\theta},
\end{equation}
so it remains to bound $\left\lVert W\right\rVert _{\infty;I}$ and $\left\lVert W\right\rVert _{\alpha\beta;I}$.

On the one hand, since $F$ is $(\alpha,\beta)$-Hölder continuous, we have
\begin{eqnarray}
\left\lVert W\right\rVert _{\infty;I} & = & \sup_{t\in\left[t_{0},t_{1}\right]}\left|F\left(t,Y\right)-F\left(t,Z\right)\right|\nonumber \\
 & \le & \tilde{c}_{\alpha,\beta}\left(\left\lVert Y-Z\right\rVert _{\alpha,\left[0,t_{1}\right]}+\left|Y_{0}-Z_{0}\right|\right)^{\beta}.\label{w2}
\end{eqnarray}
On the other hand, using \eqref{eq:Comp1} and recalling 
$\left\lVert Y\right\rVert _{\alpha;I},\left\lVert Z\right\rVert _{\alpha;I}\le R$, it follows
\begin{eqnarray}
\left\lVert W\right\rVert _{\alpha\beta;I} & \le & \left\lVert F\left(t,Y\right)\right\rVert _{\alpha\beta;I}+\left\lVert F\left(t,Z\right)\right\rVert _{\alpha\beta;I} \nonumber \\
 & \le & c_{\alpha,\beta}\left(1+\left\lVert Y\right\rVert _{\alpha;I}^{\beta}\right)+c_{\alpha,\beta}\left(1+\left\lVert Z\right\rVert _{\alpha;I}^{\beta}\right)\nonumber\\
 & \le & 2c_{\alpha,\beta}\left(1+R^{\beta}\right). \label{w3}
\end{eqnarray}
Finally inequality \eqref{eq:Comp2} follows  substituting  \eqref{w2} and \eqref{w3} into \eqref{w1}, recalling that 
$W_{t}:=F\left(t,Y\right)-F\left(t,Z\right), t \in [0,T] $.

\end{proof}

\section{\label{sec:Existence-and-Uniqueness}The Existence Results}

In this section we present the main results of this paper. We introduce
formally the equation and its statements regarding existence of solutions.
We state a version of Schauder fixed point theorem that we are using. 
\begin{thm}
\label{thm:SCHAUDER}Let $M$ be a non-empty, closed, bounded, convex
subset of a Banach space, and suppose $S:M\rightarrow M$ is a continuous
operator which maps $M$ into a compact subset of $M$. Then M has
a fixed point.
\end{thm}
\begin{proof}
See \cite[Th. 2.2]{schauderbonsall1962lectures}. 
\end{proof}
The next two lemmas will will help us gluing Hölder functions. 
\begin{lem}
\label{lem:Gluing_holder} Let $I$ and $J$ denote two compact intervals
of $\mathbb{R}$ such that $I\cap J$ is non-empty. Let $Y:I\cup J\rightarrow U$
a path such that $Y\in C^{\alpha}\left(I\right)$ and $Y\in C^{\alpha}\left(J\right)$.
Then $Y\in C^{\alpha}\left(I\cup J\right)$ and 
\[
\left\lVert Y\right\rVert _{\alpha;I\cup J}\le2\left(\left\lVert Y\right\rVert _{\alpha;I}+\left\lVert Y\right\rVert _{\alpha;J}\right).
\]
\end{lem}
\begin{proof}
See \cite[Lemma 3]{gubinelli2004controlling}. \end{proof}
\begin{lem}
\label{lem:Glueing_holder_II} Fix $\tau>0$. Let $0=:t_{0}<t_{1}<\cdots<t_{N+1}:=T$
be a partition of $\left[0,T\right]$, where every sub-interval has
length $\tau$, i.e., $t_{i}-t_{i-1}=\tau$, $\forall i=1,\ldots,N$.
Let $Y:\left[0,T\right]\rightarrow U$ such that 
\[
\left\lVert Y\right\rVert _{\alpha;t_{i},\tau}\le R,
\]
for all $i=1,\ldots,N$. Then 
\[
\left\lVert Y\right\rVert _{\alpha;\left[0,T\right]}\le4R\left(1\vee T^{1-\alpha}\tau^{\alpha-1}\right).
\]
\end{lem}
\begin{proof}
Let $s,t\in\left[0,T\right]$. If $\left|t-s\right|\le\tau$ then
$s$ and $t$ belong to (at most) two consecutive interval $\left[t_{i},t_{i+1}\right]$.
Hence, $\left|Y_{t}-Y_{s}\right|\le2(R+R)\left|t-s\right|^{\alpha}$,
see Lemma \ref{lem:Gluing_holder} above. Otherwise, $t_{K}<s\le t_{K+1}<\cdots<t<t_{K+r+1}$
with $\tau r<t-s\le\tau\left(r+1\right)$ for some $r\ge 1$. Then 
\begin{align*}
\left|Y_{t}-Y_{s}\right| & \le\sum_{j=1}^{r+1}\left|Y_{t\wedge t_{K+j}}-Y_{s\vee t_{K+j-1}}\right|\\
 & \le R\sum_{j=1}^{r}\left|t\wedge t_{K+j}-s\vee t_{K+j-1}\right|^{\alpha}\\
 & \le R\tau^{\alpha}\left(r+1\right)\\
 & =R\tau^{\alpha-1}\tau (r+1)\\
 & \le 2R\tau^{\alpha-1}\tau r \\
 & \le 2R\tau^{\alpha-1}\left|t-s\right|\\
 & \le 2R\tau^{\alpha-1}T^{1-\alpha}\left|t-s\right|^{\alpha}.
\end{align*}

In conclusion $\left\lVert Y\right\rVert _{\alpha;\left[0,T\right]}\le4R\left(1\vee T^{1-\alpha}\tau^{\alpha-1}\right)$. 
\end{proof}

Now we state and prove the first existence theorem for global solutions
  in time. We insist  on the fact that our assumptions
do not imply that $F$ is bounded. This particular case
will be investigated in the subsequent Theorem \ref{T20}.
%
%
\begin{thm}
\label{thm:F_is_sublinear}
 Let $U$ and $V$ be finite dimensional linear spaces.
 Let $\alpha>\frac{1}{2}$ and $X:\left[0,T\right]\rightarrow V$
an $\alpha$-Hölder path. Let $\beta\in\left(0,1\right)$ such $\alpha\beta+\alpha>1$.
Let $F:\left[0,T\right]\times C^{\alpha'}\left(\left[0,T\right],U\right)\rightarrow L\left(V,U\right)$
for some $\alpha'<\alpha$ with the property $F$ is also $\left(\alpha',\beta\right)$-Hölder
continuous. \\
 Given an initial condition $y_{0}\in U,$ there is a solution $Y\in C^{\alpha}\left(\left[0,T\right],U\right)$
for the equation  
\begin{equation}
Y_{t}=y_{0}+\int_{0}^{t}F\left(u,Y\right)dX_{u}, \ t\in\left[0,T\right].
\label{eq:EDO}
\end{equation}
\end{thm}

\begin{rem}
\label{rem_sublinear} 
\begin{enumerate}
\item A more general framework of \eqref{eq:EDO} 
is a path-dependent equation with initial condition at 
$t_0$ instead of $0$, for $t \in [t_0,t_1]$ with $t_1 \in [t_0,T]. $ 
 In that case the initial condition
will be a function  $\eta\in C^{\alpha}[0,t_{0}]$.

In correspondence to this we introduce $Z^{\eta}:[0,t_{1}]\rightarrow U$ setting 
\begin{equation}
Z_{t}^{\eta}:=\begin{cases}
\eta_{t}; & t\in\left[0,t_{0}\right]\\
Z_{t}; & t\in\left[t_{0},t_{1}\right]. 
\end{cases}\label{eq:Z_eta}
\end{equation}
The equation of our interest is
\begin{equation}
\begin{cases}
Y_{t}=\eta_{t_{0}}+\int_{t_{0}}^{t}F\left(u,Y^{\eta}\right)dX_{u}, & t\in\left[0,t_{1}\right],\\
Y_{s}=\eta_{s}, & s\in[0,t_{0}].
\end{cases}\label{eq:EDObis}
\end{equation}
We remark that \eqref{eq:EDO} is a particular case of \eqref{eq:EDObis}
setting $t_0 = 0$ and $t_1 = T$.

\item The strategy employed in the proof will be first 
 to construct a solution
of \eqref{eq:EDO} replacing $T$ with a \textit{small} time $\tau$.
Then given $t_{0}$, which will be of the type $t_{0}=k\tau$ for
$k=0,1,\ldots,$, and $\eta\in C^{\alpha}([0,t_{0}])$ we will inductively
construct a solution of \eqref{eq:EDObis} with $t_{1}=t_{0}+\tau$.
\\
 For the general induction step we consider the so called \textit{solution
map}, i.e. a functional $S_{\eta}:M\rightarrow C^{\alpha'}([t_{0},t_{1}])$,
where $M$ is a suitable subset of $C^{\alpha'}([t_{0},t_{1}])$ which
will be introduced later, so that $S_{\eta}(M)\subset M$, defined
by 
\begin{equation} \label{eq:DEF_S_bounded} 
S_{\eta}\left(Z\right)_{t}:=\eta_{t_{0}}+\int_{t_{0}}^{t}F\left(u,Z^{\eta}\right)dX_{u}, \ t\in\left[t_{0},t_{0}+\tau\right]
\end{equation}
 and $Z\in M$. 
We will prove
that $S_{\eta}$ has a fixed point through Schauder's Theorem \ref{thm:SCHAUDER}),
which of course solves  (\ref{eq:EDObis}) in $\left[t_{0},t_{1}\right]$.
This would imply
the existence of a solution on the whole time interval $\left[0,T\right]$,
by patching solutions together.

\end{enumerate}
\end{rem}

\begin{proof}
\textbf{Step 1. }\emph{We can assume without loss of generality that}
\begin{equation}
\alpha'\beta+\alpha'>1,\label{eq:alpha_prime}
\end{equation}
\emph{and moreover, that there is $\theta\in\left(0,1\right)$ such that
$F$ is $\left(\alpha'\theta^{2},\beta\right)$-Hölder continuous
and} 
\begin{equation}
\alpha'\theta^{2}\beta+\alpha'\theta^{2}>1,\label{eq:alpha_prime_theta}
\end{equation}
\emph{as well}. \\
 For this we are going to fabricate constants $\tilde{\alpha}$ larger
than $\alpha'$ and $\theta$ such that $F$ is $\left(\tilde{\alpha}\theta^{2},\beta\right)$-Hölder
continuous and $(\tilde{\alpha},\theta)$ fulfill  \eqref{eq:alpha_prime}
and \eqref{eq:alpha_prime_theta} with $\alpha'$ replaced by $\tilde{\alpha}$.
\\
Indeed, since by hypothesis, $\alpha$ is strictly greater than $\frac{1}{2}$
and because of the inequality $\alpha\beta+\alpha>1$, we can first
choose $\tilde{\alpha}\in\left(\frac{1}{\beta+1},\alpha\right)\cap\left(\frac{1}{2},\alpha\right)$.
$\theta\in\left(0,1\right)$ such that $\theta^{2}\tilde{\alpha}\in\left(\frac{1}{\beta+1},\tilde{\alpha}\right)\cap\left(\frac{1}{2},\tilde{\alpha}\right)$,
which is possible by a similar reasoning. Now, we have $\alpha'<\theta^{2}\tilde{\alpha}<\tilde{\alpha}<\alpha$
so Lemma \ref{lem:F_is_alpha_prime_holder} guarantees that $F$ is
$\left(\tilde{\alpha}\theta^{2},\beta\right)$-Hölder continuous,
with $\tilde{\alpha}$ and $\theta$ fulfilling the inequalities  \eqref{eq:alpha_prime}
and \eqref{eq:alpha_prime_theta} with $\alpha'$ replaced with $\tilde{\alpha}$.

Morally this step consists in restricting the domain of the vector field $F$ into a suitable smaller set $C^{\tilde{\alpha}}$, i.e.,  $C^{\alpha} \subset C^{\tilde{\alpha}} \subset C^{\alpha'}$. The choice of a suitable $\tilde{\alpha}$ does not play any role regarding the space where the solution $Y$ lives, as we can see in the next step. 

\textbf{Step 2.} \emph{Looking for a solution of \eqref{eq:EDObis}
in $C^{\alpha}([t_{0},t_{1}])$, it is enough to show that there is
a solution $Y\in C^{\alpha'}([t_{0},t_{1}])$. }

Indeed, if $Y\in C^{\alpha'}([t_{0},t_{1}])$ solves  
(\ref{eq:EDObis}), then by   (\ref{eq:Young_inequality_general})
in Theorem \ref{thm:Young_integral} and   (\ref{eq:F_is_spec_holder}),
it follows 
\begin{align*}
\left|Y_{t}-Y_{s}\right| & =\left|\int_{s}^{t}F\left(u,Y\right)dX_{u}\right|\\
 & \le k_{\alpha'\beta+\alpha}c\left(1+\left\lVert Y\right\rVert _{\alpha';\left[s,t\right]}^{\beta}\right)\left\lVert X\right\rVert _{\alpha;\left[0,T\right]}\left|t-s\right|^{\alpha'\beta+\alpha}+\left|F\left(s,Y\right)\right|\left\lVert X\right\rVert _{\alpha;\left[0,T\right]}\left|t-s\right|^{\alpha}\\
 & \le\left\{ k_{\alpha'\beta+\alpha}c\left(1+\left\lVert Y\right\rVert _{\alpha';\left[0,t_{1}\right]}^{\beta}\right)T^{\alpha'\beta}+\left\lVert F\left(\centerdot,Y\right)\right\rVert _{\infty}\right\} \left\lVert X\right\rVert _{\alpha;\left[0,T\right]}\left|t-s\right|^{\alpha}, 
\end{align*}
where $c = c_{\alpha', \beta}$ as defined in \eqref{eq:F_is_spe_lip}. 

\textbf{Step 3.} \emph{Discussion about set the 
 $M\subset C^{\alpha'}([t_{0},t_{0}+\tau])$,
anticipated in the Remark \ref{rem_sublinear} item 2.} $M$ will be of
the type 
\begin{equation}
M:=M_{t_{0},\tau,R,a}^{\alpha'}:=\left\{ Z\in C^{\alpha'}([t_{0},t_{0}+\tau])\ | \ Z_{t_{0}}=a,\mbox{ and }\left\Vert Z\right\Vert _{\alpha';t_{0},\tau}\le R\right\} ,\label{eq:DEF_M}
\end{equation}
for fixed $R,\tau>0$ and $a\in U$. 
We will indeed set $a=\eta_{t_{0}}\in U$, and $R,\tau$ will be suitable parameters,
see  (\ref{eq:tau_condition_II}) and \eqref{eq:R_cond}, in order
to guarantee 
that $S_{\eta}\left(M_{t_{0},\tau,R,\eta_{t_{0}}}^{\alpha'}\right)\subset M_{t_{0},\tau,R,\eta_{t_{0}}}^{\alpha'}$.

\textbf{Step 4.} \emph{Let $R>0$, $t_{0}\in[0,T)$, $\tau\in(0,1]$,
$a\in U$, which will be arbitrary in this step, the set $M_{t_{0},\tau,R,a}^{\alpha'}$
is compact in $C^{\theta\alpha'}$-topology.} This is a standard result,
however though we present its proof for the sake of completeness.

We recall the set $M_{t_{0},\tau,R,a}^{\alpha'}$ is a subset of 
$C^{\theta\alpha'}([t_{0},t_{0}+\tau])$
with a fixed initial condition, hence $C^{\theta\alpha'}$-topology
in $M_{t_{0},\tau,R,a}^{\alpha'}$ is induced by $\left\Vert \centerdot\right\Vert _{\alpha'\theta;t_{0},\tau}$.
Now we prove the claim, let $(Z^{n})$ be a sequence in $M_{t_{0},\tau,R,a}^{\alpha'}$,
$n\in\mathbb{N}$. This is an equicontinuous family in $[t_0, t_0 +\tau]$, since $\left|Z_{t}^{n}-Z_{s}^{n}\right|\le R\left|t-s\right|^{\alpha'},\ s,t\in[t_{0},t_{0}+\tau]$
for each $n$. Also, the sequence is uniformly bounded since $\left|Z_{t}^{n}\right|\le\left|Z_{0}^{n}\right|+R\left|t-0\right|^{\alpha'}=\left|a\right|+R\left|t\right|^{\alpha'}$.
Hence, by the classical Arzelà-Ascoli Theorem, there is $Z\in C([t_0,t_0+\tau])$
such that, for a subsequence (also denoted by $Z^{n}$), $\left\Vert Z^{n}-Z\right\Vert _{\infty;t_{0},t_{0}+\tau}\rightarrow0$,
as $n\rightarrow\infty$. So, in particular $Z_{t_{0}}=a$, on the
one hand. On the other hand, it is easy to see that $\left\Vert Z\right\Vert _{\alpha';t_{0},\tau}\le R$,
since 
\begin{align*}
\left|Z_{t}-Z_{s}\right| & =\lim_{n\rightarrow\infty}\left|Z_{t}^{n}-Z_{s}^{n}\right|\\
 & \le\lim_{n\rightarrow\infty}R\left|t-s\right|^{\alpha'}.
\end{align*}
Therefore $Z\in M_{t_{0},\tau,R,a}^{\alpha'}$.

In order to show $\left\Vert Z^{n}-Z\right\Vert _{\theta\alpha';t_{0},\tau}\underset{n\rightarrow\infty}{\rightarrow}0$,
we use the geometric interpolation, see Remark 
\ref{rem:Geometric_interpolation}, so
\begin{align*}
\left\Vert Z^{n}-Z\right\Vert _{\theta\alpha'} & \le2^{1-\theta}\left\Vert Z^{n}-Z\right\Vert _{\infty;t_{0},\tau}^{1-\theta}\left\Vert Z^{n}-Z\right\Vert _{\alpha';t_{0},\tau}^{\theta}\\
 & \le2^{1-\theta}\left\Vert Z^{n}-Z\right\Vert _{\infty;t_{0},\tau}^{1-\theta}\left(\left\lVert Z^{n}\right\rVert _{\alpha';t_{0},\tau}+\left\lVert Z\right\rVert _{\alpha';t_{0},\tau}\right)^{\theta}\\
 & \le2^{1-\theta}\left\Vert Z^{n}-Z\right\Vert _{\infty;t_{0},\tau}^{1-\theta}\left(2R\right)^{\theta}
\end{align*}
and we observe that the right-hand side converges to zero as $n$ goes
to $\infty$.

\textbf{Step 5.}\emph{ Let $R>0$, $t_{0}\in[0,T)$, $\tau\in(0,1]$,
$a\in U$ and $\eta\in C^{\alpha}([t_{0},t_{0}+\tau])$,
which will be arbitrary in this step. Let $\theta$ as introduced
in Step 1. Then the map $S_{\eta}:M_{t_{0},\tau,R,a}^{\alpha'}\rightarrow
 C^{\alpha'}([t_{0},t_{0}+\tau])$
is continuous under the $C^{\alpha'\theta}([t_{0},t_{0}+\tau])$-topology. We recall that $\alpha'\theta^{2}\beta+\alpha'\theta^{2}>1$, see  \eqref{eq:alpha_prime_theta}.}
%
%
Since $M_{t_{0},\tau,R,a}^{\alpha'}$ is compact, it is a closed subset
of $C^{\alpha'\theta}([t_{0},t_{0}+\tau])$. 

Fix an arbitrary $W\in M_{t_{0},\tau,R,a}^{\alpha'}$. We will show
$\left\Vert S_{\eta}\left(Z\right)-S_{\eta}\left(W\right)\right\Vert _{\theta\alpha';t_{0},\tau}\rightarrow0$
as $\left\Vert Z-W\right\Vert _{\theta\alpha';t_{0},\tau}\rightarrow0$,
$Z\in M_{t_{0},\tau,R,a}^{\alpha'}$.

We use now Young integral inequality   (\ref{eq:Young_inequality})
with $\alpha$ replaced with $\alpha'\theta$ and $\gamma$ replaced
with $\alpha'\theta^{2}\beta$. We observe that the sum 
$\mu :=\alpha'\theta^{2}\beta+\alpha'\theta^{2}$ which is strictly larger
than $1$ so Theorem \ref{thm:Young_integral} can be applied and
by   (\ref{eq:Young_inequality}) 
\begin{align}
\left\Vert S_{\eta}\left(Z\right)-S_{\eta}\left(W\right)\right\Vert _{\theta\alpha';t_{0},\tau} & =\left\Vert \int_{t_{0}}^{\centerdot}F\left(u,Z^{\eta}\right)-F\left(u,W^{\eta}\right)dX_{u}\right\Vert _{\theta\alpha';t_{0},\tau}\nonumber \\
 & \le\left(k_{\mu}+1\right)\left\Vert X\right\Vert _{\theta\alpha';t_{0},\tau}\left\Vert F\left(\centerdot,Z^{\eta}\right)-F\left(\centerdot,W^{\eta}\right)\right\Vert _{\theta^{2}\alpha'\beta;t_{0},\tau}.\label{eq:S4}
\end{align}
We recall that $F$ is $(\alpha'\theta^{2},\beta)$-Hölder continuous,
see Step 1.

From Proposition \ref{prop:FoY_is_continuous}, see  \eqref{eq:Comp2},
using $\alpha'\theta$ (and $t_{0}+\tau$) instead of $\alpha$ (and
$t_{1}$), it follows that 
\begin{align}
\left\Vert F\left(\centerdot,Z^{\eta}\right)-F\left(\centerdot,W^{\eta}\right)\right\Vert _{\theta^{2}\alpha'\beta;t_{0},\tau} & \le2\tilde{c}_{\alpha'\theta^{2},\beta}^{1-\theta}c_{\alpha'\theta^{2},\beta}^{\theta}\left(1+R^{\beta}\right)^{\theta}\left\Vert Z^{\eta}-W^{\eta}\right\Vert _{\theta\alpha';0,t_{0}+\tau}^{\beta\left(1-\theta\right)}\nonumber \\
 & =2\tilde{c}_{\alpha'\theta^{2},\beta}^{1-\theta}c_{\alpha'\theta^{2},\beta}^{\theta}\left(1+R^{\beta}\right)^{\theta}\left\Vert Z-W\right\Vert _{\theta\alpha';t_{0},\tau}^{\beta\left(1-\theta\right)}.\label{eq:S5}
\end{align}
From   (\ref{eq:S4}) and (\ref{eq:S5}) we conclude that $S_{\eta}$
is continuous with respect to $C^{\alpha'\theta}$-topology.

\textbf{Step 6.} \emph{We prove now that $S_{\eta}\left(M_{t_{0},\tau,R,y_{0}}^{\alpha'}\right)\subset M_{t_{0},\tau,R,y_{0}}^{\alpha'}$
in the case when $t_{0}=0$, with $\eta:\{0\}\rightarrow U$, $\eta_{0}=y_{0}$
for any $y_{0}\in U$} \emph{and suitable $R,\tau>0$ introduced below}.
We will extend this property at Step 8. for $t_{0}=N\tau$, $N=1,2,\ldots$.

We set 
\begin{equation} \label{EStep6} 
K:=\left(k_{\alpha'\beta+\alpha'}+1\right)\left\Vert X\right\Vert _{\alpha;0,T}2c_{F}\left(1+T^{\alpha'\beta}\right)
\end{equation}
and $c_{F}:=\max\left\{ \left|F\left(0,0\right)\right|;c_{\alpha',\beta};\tilde{c}_{\alpha',\beta}\right\} $.
We can assume $K>0$, otherwise either $\left\Vert X\right\Vert _{\alpha;0,T}=0$
or $F=0$, thus the constant function $Y_{t}:=y_{0}$ solves 
(\ref{eq:EDO}). Let $\varepsilon\in\left(0,\frac{K}{2}\right)$ be
fixed and $\tau$ is defined by 
\begin{equation}
\tau:=\left(\frac{\varepsilon}{K}\right)^{\frac{1}{\alpha-\alpha'}},\label{ETau}
\end{equation}
so that, $0<\tau<1$ and 
\begin{equation}
K\tau^{\alpha-\alpha'}=\varepsilon.\label{eq:tau_condition_II}
\end{equation}

Let $R>0$ big enough such that 
\begin{align}
\varepsilon\left(1+5\left(1\vee T^{1-\alpha'}\tau^{\alpha'-1}\right)^{\beta}R^{\beta}\right) & \le R,\label{eq:R_cond}\\
\left|y_{0}\right| & \le R, \label{eq:R_cond1}
\end{align}
which is always possible since $\beta<1$. Indeed, given a function
$g:\mathbb{R}\rightarrow\R$ defined by $g(R)=c+dR^{\beta}$, $c,d>0$,
the limit of $\frac{g(R)}{R}$ when $R\rightarrow\infty$ is zero.

From now on in this step we set $S:=S_{\eta}$. We prove now that
$S\left(M_{0,\tau,R,y_{0}}^{\alpha'}\right)\subset M_{0,\tau,R,y_{0}}^{\alpha'}$.
Indeed, given $Z\in M_{0,\tau,R,y_{0}}^{\alpha'}$, from Young integral
inequality  (\ref{eq:Young_inequality}), it follows that 
\begin{align}
\left\Vert S\left(Z\right)\right\Vert _{\alpha';0,\tau} & \le\left(k_{\alpha'\beta+\alpha'}+1\right)\left\Vert X\right\Vert _{\alpha';0,\tau}\left(\left\Vert F\left(\centerdot,Z\right)\right\Vert _{\alpha'\beta;0,\tau}+\left|F\left(0,Z\right)\right|\right).\label{eq:S1}
\end{align}

Since $F$ is $\left(\alpha',\beta\right)$-Hölder continuous, by
Proposition \ref{prop:FoY_is_continuous} it follows that 
\begin{align}
\left\Vert F\left(\centerdot,Z\right)\right\Vert _{\alpha'\beta;0,\tau} & \le c_{\alpha',\beta}\left(1+\left\Vert Z\right\Vert _{\alpha';0,\tau}^{\beta}\right).\label{eq:S2}
\end{align}
Moreover by   (\ref{eq:F_is_spe_lip}), it also holds that 
\begin{align}
\left|F\left(0,Z\right)\right| & \le\left|F\left(0,0\right)\right|+\tilde{c}_{\alpha',\beta}\left|Z_{0}-0\right|^{\beta}\nonumber \\
 & =\left|F\left(0,0\right)\right|+\tilde{c}_{\alpha',\beta}\left|y_{0}\right|^{\beta}.\label{eq:S3}
\end{align}
Plugging   (\ref{eq:S2}),(\ref{eq:S3}) into   (\ref{eq:S1}),
using $\left\Vert X\right\Vert _{\alpha';0,\tau}\le\left\Vert X\right\Vert _{\alpha;0,T}\tau^{\alpha-\alpha'}$
and $\left|y_{0}\right|,\left\Vert Z\right\Vert _{\alpha';0,\tau}\le R$,
also recalling $c_{F}=\max\left\{ \left|F\left(0,0\right)\right|;c_{\alpha',\beta};\tilde{c}_{\alpha',\beta}\right\} $,
definitions from $\tau$ and $R$, (see   (\ref{eq:tau_condition_II}),
(\ref{eq:R_cond})), we have
\begin{align*}
\left\Vert S\left(Z\right)\right\Vert _{\alpha';0,\tau} & \le\left(k_{\alpha'\beta+\alpha'}+1\right)\left\Vert X\right\Vert _{\alpha;0,T}\tau^{\alpha-\alpha'}\left(c\left(1+R^{\beta}\right)+\left|F\left(0,0\right)\right|+\tilde{c}R^{\beta}\right)\\
 & \le\left(k_{\alpha'\beta+\alpha'}+1\right)\left\Vert X\right\Vert _{\alpha;0,T}\tau^{\alpha-\alpha'}\left(2c_{F}\left(1+R^{\beta}\right)\right)\\
 & \le K\tau^{\alpha-\alpha'}\left(1+R^{\beta}\right)\\
 & =\varepsilon\left(1+R^{\beta}\right)\\
 & \le R.
\end{align*}
This proves Step 6.

\medskip
Let $R > 0$ as in \eqref{eq:R_cond} and \eqref{eq:R_cond1} together
with $\tau$ selected in \eqref{ETau} until the end of the proof.

\textbf{Step 7.} \emph{There is a solution $Y\in C^{\alpha}([0,\tau])$
for   (\ref{eq:EDO}) replacing $T$ with $\tau$,
with $Y_{0}=y_{0}$ and $\left\Vert Y\right\Vert _{\alpha';0,\tau}\le R$.
This constitutes the first stage of a statement which will be proved
by induction in Step 9. below.} \\
 This simply follows from Steps 4., 5., 6. which allow us to use Theorem
\ref{thm:SCHAUDER} and finally Step 2.

From Step 6. the map
 $S:M_{0,\tau,R,y_{0}}^{\alpha'}\rightarrow M_{0,\tau,R,y_{0}}^{\alpha'}$
is well-defined and Step 5. shows us it is continuous under $C^{\alpha'\theta}$-topology.
Since $M_{0,\tau,R,y_{0}}^{\alpha'}$ is compact under $C^{\alpha'\theta}$-topology,
see Step 4., Schauder's Theorem \ref{thm:SCHAUDER} claims that there
is a fixed point for the map $S$, denoted by $Y\in M_{0,\tau,R,a}^{\alpha'}$.
In other words, there is $Y\in M_{0,\tau,R,a}^{\alpha'}$, such that
\[
Y_{t}=S\left(Y\right)_{t}=y_{0}+\int_{0}^{t}F\left(u,Y\right)dX_{u}, \ t\in\left[0,\tau\right].
\]
Finally, from Step 2., we conclude that
$Y\in C^{\alpha}\left[0,\tau\right]$.

\textbf{Step 8.} \emph{Now we prove the general statement announced
in step 6. Let $t_{0}=N\tau$ for some $N=1,2,\ldots$. Assume that
 $\eta\in C^{\alpha}([0,t_{0}])$
such that $\left\Vert \eta\right\Vert _{\alpha';k\tau,\tau}\le R$
for $k=0,1,\ldots,N-1$, we have
\begin{equation}
S_{\eta}\left(M_{t_{0},\tau,R,\eta_{t_{0}}}^{\alpha'}\right)\subset M_{t_{0},\tau,R,\eta_{t_{0}}}^{\alpha'}.
\label{EStability}
\end{equation}
} Indeed, let $Z\in M_{t_{0},\tau,R,\eta_{t_{0}}}^{\alpha'}$. From
Young integral inequality  (\ref{eq:Young_inequality}) using $\alpha'\beta$
(and $\alpha'$) instead of $\delta$ (and $\alpha$), 
\begin{align}\label{eq:Sgen}
\left\Vert S_{\eta}\left(Z\right)\right\Vert _{\alpha';t_{0},\tau} & \le\left(k_{\alpha'\beta+\alpha'}+1\right)\left\Vert X\right\Vert _{\alpha';t_{0},\tau}\left\{ \left\Vert F\left(\centerdot,Z^{\eta}\right)\right\Vert _{\alpha'\beta;t_{0},\tau}+\left|F\left(t_{0},Z^{\eta}\right)\right|\right\} .
\end{align}
Since $F$ is $\left(\alpha',\beta\right)$-Hölder continuous, by
  (\ref{eq:Comp1}) from Proposition \ref{prop:FoY_is_continuous}
with $\alpha'$ instead of $\alpha$
and noting that $\left\Vert Z\right\Vert _{\alpha';t_0,\tau}\le R$
it follows 
\begin{align}
\left\Vert F\left(\centerdot,Z^{\eta}\right)\right\Vert _{\alpha'\beta;t_{0},\tau} & \le c_{\alpha',\beta}\left(1+\left\Vert Z^{\eta}\right\Vert _{\alpha';t_{0},\tau}^{\beta}\right)\nonumber \\
 & \le c_{\alpha',\beta}\left(1+\left\Vert Z\right\Vert _{\alpha';t_{0},\tau}^{\beta}\right)\nonumber \\
 & \le c_{\alpha',\beta}\left(1+R^{\beta}\right).\label{eq:Sgen-1}
\end{align}
Regarding $\left|F\left(t_{0},Z^{\eta}\right)\right|$, we split it
as 
\begin{align}
\left|F\left(t_{0},Z^{\eta}\right)\right| & =\left|F\left(t_{0},\eta\right)\right|\nonumber \\
 & \le\left|F\left(t_{0},\eta\right)-F\left(0,\eta\right)\right|+\vert F\left(0,\eta\right)\vert\nonumber \\
 & =:A+B.\label{eq:Sgen-2}
\end{align}
On the one hand, by   (\ref{eq:Comp1}) from Proposition \ref{prop:FoY_is_continuous}, 
Lemma \ref{lem:Glueing_holder_II} together with hypothesis $\left\Vert \eta\right\Vert _{\alpha';k\tau,\tau}\le R$,
recalling again that $c_{F}=\max\left\{ \left|F\left(0,0\right)\right|;c_{\alpha',\beta};\tilde{c}_{\alpha',\beta}\right\} $
and also that $t_{0}=N\tau\le T$ and $4^{\beta}<5$, it follows that
\begin{align}
A & =\left|F\left(N\tau,\eta\right)-F\left(0,\eta\right)\right|\nonumber \\
 & \le c_{\alpha',\beta}\left(1+\left\Vert \eta\right\Vert _{\alpha';0,N\tau}^{\beta}\right)\left(N\tau\right)^{\alpha'\beta}\nonumber \\
 & \le c_{\alpha',\beta}\left(1+\left(4R\left(1\vee T^{1-\alpha'}\tau^{\alpha'-1}\right)\right)^{\beta}\right) T^{\alpha'\beta}\nonumber \\
 & \le c_{F}T^{\alpha'\beta}\left(1+5R^{\beta}\left(1\vee T^{1-\alpha'}\tau^{\alpha'-1}\right)^{\beta}\right).\label{eq:Sgen-2a}
\end{align}
%
%

On the other hand, since $F$ is $\left(\alpha',\beta\right)$-Hölder continuous using   $5^{\beta}\le5$ and the hypothesis $\left\Vert \eta\right\Vert _{\alpha';k\tau,\tau}\le R$, $\left|y_{0}\right|\le R$ it follows that

\begin{align}
B & \le\left|F\left(0,0\right)\right|+\left|F\left(0,\eta\right)-F\left(0,0\right)\right|\nonumber \\
 & \le F\left(0,0\right)+\tilde{c}_{\alpha',\beta}\left(\left\Vert \eta-0\right\Vert _{\alpha';0,N\tau}+\left|\eta_{0}-0\right|\right)^{\beta}\nonumber \\
 & \le c_{F}\left(1+\left(\left\Vert \eta\right\Vert _{\alpha';0,N\tau}+\left|y_{0}\right|\right)^{\beta}\right)\nonumber \\
 & \le c_{F}\left(1+\left(4R\left(1\vee T^{1-\alpha'}\tau^{\alpha'-1}\right)+R\right)^{\beta}\right)\nonumber \\
 & \le c_{F}\left(1+\left(5R\left(1\vee T^{1-\alpha'}\tau^{\alpha'-1}\right)\right)^{\beta}\right)\nonumber \\
 & \le c_{F}\left(1+5R^{\beta}\left(1\vee T^{1-\alpha'}\tau^{\alpha'-1}\right)^{\beta}\right),\label{eq:Sgen-2b}
\end{align}
where we have used Lemma \ref{lem:Glueing_holder_II} in the fourth inequality.
Hence, from   (\ref{eq:Sgen-2a}) and (\ref{eq:Sgen-2b}) we conclude
\begin{equation}
\left|F\left(t_{0},Z^{\eta}\right)\right|\le c_{F}\left(1+T^{\alpha'\beta}\right)\left(1+5R^{\beta}\left(1\vee T^{1-\alpha'}\tau^{\alpha'-1}\right)^{\beta}\right).\label{eq:Sgen-3}
\end{equation}

Now, plugging   (\ref{eq:Sgen-1}) and (\ref{eq:Sgen-3}) into  
(\ref{eq:Sgen}), using $\left\Vert X\right\Vert _{\alpha';t_{0},\tau}\le\left\Vert X\right\Vert _{\alpha;0,T}\tau^{\alpha-\alpha'}$ and 
recalling the definitions of $\tau$ and $R$ given in   (\ref{eq:tau_condition_II}),
(\ref{eq:R_cond}), it yields 
\begin{align*}
\left\Vert S_{\eta}\left(Z\right)_{t}\right\Vert _{\alpha';0,\tau} & \le\left(k_{\alpha'\beta+\alpha'}+1\right)\left\Vert X\right\Vert _{\alpha;0,T}\tau^{\alpha-\alpha'}\left\{ 2c_{F}\left(1+T^{\alpha'\beta}\right)\left(1+5R^{\beta}\left(1\vee T^{1-\alpha'}\tau^{\alpha'-1}\right)^{\beta}\right)\right\} \\
 & =\varepsilon\left(1+5R^{\beta}\left(1\vee T^{1-\alpha'}\tau^{\alpha'-1}\right)^{\beta}\right) \le R.
\end{align*}
This proves   \eqref{EStability}.

\textbf{Step 9.} \emph{There is a solution
 $Y\in C^{\alpha}([0,(N+1)\tau])$
for the   \eqref{eq:EDO} replacing $T$ with $(N+1)\tau$, with
$Y_{0}=y_{0}$ and $\left\Vert Y\right\Vert _{\alpha';k\tau,\tau}\le R$,
each $k=0,1,\ldots,N$. This constitutes the induction stage which
we announced in Step 7.}

Indeed, the case $N=0$ was proved in Step 7. Now, assume there is
a solution $\eta\in C^{\alpha}([0,N\tau])$ (replacing $T$
with $N\tau$) of \eqref{eq:EDO} with $\left\Vert \eta\right\Vert _{\alpha';k\tau,\tau}\le R$,
each $k=0,1,\ldots,N-1$. The solution $\eta$ fulfills the conditions
of the Step 8. with $t_{0}=N\tau$, hence $S_{\eta}\left(M_{t_{0},\tau,R,\eta_{t_{0}}}^{\alpha'}\right)\subset M_{t_{0},\tau,R,\eta_{t_{0}}}^{\alpha'}.$
Reasoning as at Step 7., the map $S_{\eta}:M_{t_{0},\tau,R,\eta_{t_{0}}}^{\alpha'}\rightarrow M_{t_{0},\tau,R,\eta_{t_{0}}}^{\alpha'}$
has a fixed point denoted by $W$, which in particular solves  
\eqref{eq:EDObis}: 
\begin{equation}
W_{t}=S_{\eta}\left(W\right)_{t}=\eta_{N\tau}+\int_{N\tau}^{t}F\left(u,W^{\eta}\right)dX_{u},\quad t\in\left[N\tau,\left(N+1\right)\tau\right],\label{eq:W}
\end{equation}
where we remind that the notation $W^{\eta}$ was introduced in  
\eqref{eq:Z_eta}.
%

Finally, we define $Y:=W^{\eta}\in C^{\alpha}([0,(N+1)\tau])$
which trivially extends $\eta$. We show below that $Y$ solves equation
  \eqref{eq:EDO}. Indeed, on the one hand for each $t\in\left[0,N\tau\right],$  recalling $F\left(u,\eta\right)=F\left(u,Y\right)$
for $u\in\left[0,N\tau\right]$ and that $\eta$ is a solution in
the interval $\left[0,N\tau\right]$, we have
\begin{align*}
Y_{t} & =W_{t}^{\eta}
  =\eta_{t}\\
 & =y_{0}+\int_{0}^{t}F\left(u,\eta\right)dX_{u}\\
 & =y_{0}+\int_{0}^{t}F\left(u,Y\right)dX_{u}.
\end{align*}
On the other hand, arguing as above and using   \eqref{eq:W}, we have, for $t\in\left[N\tau,\left(N+1\right)\tau\right]$,
\begin{align*}
Y_{t} & =W_{t}^{\eta}  =W_{t}\\
 & =\eta_{N\tau}+\int_{N\tau}^{t}F\left(u,W^{\eta}\right)dX_{u}\\
 & =y_{0}+\int_{0}^{N\tau}F\left(u,\eta\right)dX_{u}+\int_{N\tau}^{t}F\left(u,W^{\eta}\right)dX_{u}\\
 & =y_{0}+\int_{0}^{N\tau}F\left(u,Y\right)dX_{u}+\int_{N\tau}^{t}F\left(u,Y\right)dX_{u}\\
 & =y_{0}+\int_{0}^{t}F\left(u,Y\right)dX_{u}.
\end{align*}
This concludes that $Y$ is a solution to   \eqref{eq:EDO}
on the interval $\left[0,N\tau+\tau\right]$. Moreover, $\left\Vert Y\right\Vert _{\alpha';k\tau,\tau}=\left\Vert W^{\eta}\right\Vert _{\alpha';k\tau,\tau}\le R$
for $k=0,1,\ldots,N$. Indeed for $k=0,1,\ldots,N-1$, this holds by assumption
and for $k=N$ it comes from   \eqref{EStability}. Hence this establishes
the induction step and it concludes Step 9.
\end{proof}
The theorem below shows that when $F$ is bounded the coefficient
$\beta$ is also allowed to be $1$. We remark that in Theorem 
\ref{thm:F_is_sublinear} we have required that $\beta<1$. 
\begin{thm} \label{T20}
 Let $U$ and $V$ be finite dimensional linear spaces.
Let $\alpha,\beta\in(0,1]$ such that $\alpha>\frac{1}{2}$ and $\alpha\beta+\alpha>1$.
Let $X:\left[0,T\right]\rightarrow V$ be an $\alpha$-Hölder path
and $F:\left[0,T\right]\times C^{\alpha'}\left(\left[0,T\right],U\right)\rightarrow L\left(V,U\right)$
be a bounded and $\left(\alpha',\beta\right)$-Hölder continuous vector
field for some $\alpha'<\alpha$. Given an initial condition $y_{0}\in U$,
 there is a solution $Y\in C^{\alpha}\left(\left[0,T\right],U\right)$
for the equation 
\begin{equation}
Y_{t}=y_{0}+\int_{0}^{t}F\left(u,Y\right)dX_{u}, \
t\in\left[0,T\right].
\label{eq:EDO-1}
\end{equation}
 \end{thm}
\begin{proof}
This proof is simpler than the one of
 Theorem \ref{thm:F_is_sublinear} 
(where $F$ is not bounded) and we will explain here the significant changes.
 We start
re-introducing the objects. First, as explained in Step 1. of Theorem
\ref{thm:F_is_sublinear} we can assume that $\frac{1}{2}<\alpha'<\alpha$ and
also that there is $\theta\in\left(0,1\right)$ such that $F$ is
$\left(\alpha'\theta^{2},\beta\right)$-Hölder continuous and
the inequalities \eqref{eq:alpha_prime} together with
\eqref{eq:alpha_prime_theta} are still in force.

Second, we re-define the parameters $K, \varepsilon$ and $\tau$. So,
differently from \eqref{EStep6} in Step 6.,  we set
$$ K:=\left(k_{\alpha'\beta+\alpha'}+1\right)\left\Vert X\right\Vert _{\alpha;0,T}\left(2c_{\alpha',\beta}+\left\lVert F\right\rVert _{\infty}\right).$$
We  fix $\varepsilon\in\left(0,\frac{K}{2}\wedge1\right)$ and define
again $\tau$ as in \eqref{ETau} so that  $\tau\in\left(0,1\right)$ and
\eqref{eq:tau_condition_II} still holds.

For an arbitrary $t_{0}\in[0,T)$ and $\eta\in C^{\alpha}([0,t_{0}])$
with $\eta_{0}=y_{0}$, similarly as in the proof of
Theorem \ref{thm:F_is_sublinear}, we will find a solution 
of   \eqref{eq:EDObis} with $t_1 =  t_{0}+\tau $.
This can be done performing the same program of Steps 2. to 5.
  in the proof of Theorem  \ref{thm:F_is_sublinear}.
The notation $M_{t_{0},\tau,1,a}^{\alpha'}$ will denote 
the same ball as in   \eqref{eq:DEF_M}.
We define $Z^{\eta}$ as in   \eqref{eq:Z_eta} and 
$S_{\eta}:M_{t_{0},\tau,1,\eta_{t_{0}}}^{\alpha'}\rightarrow 
C^{\alpha'}([t_{0},t_{0}+\tau])$ 
as in \eqref{eq:DEF_S_bounded}. 


In order to show that $S_{\eta}$ has a fixed point (which therefore solves
  \eqref{eq:EDObis}) we need to prove 
first that 
\begin{equation} \label{SInclusion}
S_{\eta}\left(M_{t_{0},\tau,1,\eta_{t_{0}}}^{\alpha'}\right)\subset M_{t_{0},\tau,1,\eta_{t_{0}}}^{\alpha'}.
\end{equation}
This will replace
 Steps 6. and 8. of the proof of Theorem 
\ref{thm:F_is_sublinear} which
 will merge, since it will not be necessary
to distinguish $t_0 = 0$ from $t_0 = N \tau,\  N \ge 1$.

Indeed, given $Z\in M_{t_{0},\tau,1,\eta_{t_{0}}}^{\alpha'}$, from
Young integral inequality   (\ref{eq:Young_inequality}), replacing
$\gamma$ 
 there by $\alpha'\beta$ 
as for   \eqref{eq:S1},  it follows that 
\begin{align}
\left\Vert S_{\eta}\left(Z\right)\right\Vert _{\alpha';t_{0},\tau} & \le\left(k_{\alpha'\beta+\alpha'}+1\right)\left\Vert X\right\Vert _{\alpha';t_0,\tau}\left(\left\Vert F\left(\centerdot,Z^{\eta}\right)\right\Vert _{\alpha'\beta;t_{0},\tau}+\left|F\left(t_{0},Z^{\eta}\right)\right|\right).\label{eq:S1_bounded}
\end{align}
Since $F$ is $\left(\alpha',\beta\right)$-Hölder continuous, it follows
from Proposition \ref{prop:FoY_is_continuous}, see   (\ref{eq:Comp1})
with $\alpha'$ (resp. $Z^\eta$) instead of $\alpha$ (resp. $Y$) that
\begin{align}
\left\Vert F\left(\centerdot,Z^{\eta}\right)\right\Vert _{\alpha'\beta;t_{0},\tau} & \le c_{\alpha',\beta}\left(1+\left\Vert Z^{\eta}\right\Vert _{\alpha';t_{0},\tau}^{\beta}\right)\nonumber \\
 & =c_{\alpha',\beta}\left(1+\left\Vert Z\right\Vert _{\alpha';t_{0},\tau}^{\beta}\right)\nonumber \\
 & \le c_{\alpha',\beta}\left(1+1^{\beta}\right)\nonumber \\
 & \le2c_{\alpha',\beta}.\label{eq:S2_bounded}
\end{align}
Also, since $F$ is bounded it holds 
\begin{align}
\left|F\left(t_{0},Z^{\eta}\right)\right| & =\left|F\left(t_{0},\eta\right)\right|\nonumber \\
 & \le\left\lVert F\right\rVert _{\infty}.\label{eq:S3_bounded}
\end{align}
 We substitute   (\ref{eq:S2_bounded}) and (\ref{eq:S3_bounded}) into
  (\ref{eq:S1_bounded}). Using $\left\Vert X\right\Vert _{\alpha';t_{0},\tau}\le\left\Vert X\right\Vert _{\alpha;0,T}\tau^{\alpha-\alpha'}$,
 $\varepsilon\le1$
and recalling that $\tau$ satisfies   \eqref{EStep6}, it
 follows
that 
\begin{align*}
\left\Vert S_{\eta}\left(Z\right)\right\Vert _{\alpha';t_{0},\tau} & \le\left(k_{\alpha'\beta+\alpha'}+1\right)\left\Vert X\right\Vert _{\alpha;0,T}\tau^{\alpha-\alpha'}\left(2c_{\alpha',\beta}+\left\lVert F\right\rVert _{\infty}\right)\\
 & =K\tau^{\alpha-\alpha'}\\
 & =\varepsilon\\
 & \le1,
\end{align*}
henceforth $S_{\eta}\left(M_{t_{0},\tau,1,\eta_{t_{0}}}^{\alpha'}\right)\subset M_{t_{0},\tau,1,\eta_{t_{0}}}^{\alpha'}$.
This concludes the proof of   \eqref{SInclusion}.

The sequel of the proof consists in treating simultaneously
the corresponding Steps 7. and 9. of the proof of Theorem
\ref{thm:F_is_sublinear}.
We claim that $S_{\eta}:M_{t_{0},\tau,1,\eta_{t_{0}}}^{\alpha'}\rightarrow M_{t_{0},\tau,1,\eta_{t_{0}}}^{\alpha'}$
has a fixed point.   

 Indeed, we already know that $M_{t_{0},\tau,1,\eta_{t_{0}}}^{\alpha'}$
is a compact set with respect $\alpha'\theta$-Hölder topology (see
Step 4. in the proof of 
Theorem \ref{thm:F_is_sublinear}) and $S_{\eta}$ is continuous
with respect $\alpha'\theta$-Hölder topology (see Step 5. in the proof of
 Theorem
\ref{thm:F_is_sublinear}). So $S_{\eta}$ verifies the hypothesis
of Schauder's Theorem \ref{thm:SCHAUDER},
  hence there is a fixed point for $S_{\eta}$,
which we denote by $W\in C^{\alpha'}([t_{0},t_{0}+\tau])$;
$W$ also belongs to 
$ C^{\alpha}([t_{0},t_{0}+\tau])$ by Step 2. in the proof of
 Theorem \ref{thm:F_is_sublinear}.
In other words, 
\begin{equation}\label{eq:W-solves}
W_{t}=\eta_{t_{0}}+\int_{t_{0}}^{t}F\left(u,W^{\eta}\right)dX_{u}, t \in 
[t_0,t_0+\tau].
\end{equation}

Finally, we can conclude the proof showing that there is a solution
 $Y\in C^{\alpha}([0,T])$ to   \eqref{eq:EDO} proceeding similarly
as after   \eqref{eq:W}.

\end{proof}
\section{\label{sec:Particular-Case}
On a   particular path-dependent structure of the vector field.}

We conclude the paper showing that, in some cases, the solution to
our path-dependent equation can be constructed directly. In this case
the method leads us to an existence (and even uniqueness)
 statement under weaker assumptions
than in Theorems \ref{thm:F_is_sublinear} and \ref{T20}. 
This happens when the past dependence structure of the vector $F$ allows
a gap (of size $\delta>0$) between the past and the present, see
Definition \ref{def:F_is_non_ant_delta} below for the precise meaning. 
\begin{defn}
\label{def:F_is_non_ant_delta}Let $F:\left[0,T\right]\times C\left(\left[0,T\right],U\right)\rightarrow L$
be an vector field. We will say that it is \textbf{$\delta$-non-anticipating}
if it satisfies 
\begin{equation}
F\left(t,Y\right)=F\left(t,\mathcal{Y}_{(t-\delta)_{+}}\right),\label{eq:F_nonant}
\end{equation}
for all $Y\in C\left(\left[0,T\right],U\right)$ and $t\in\left[0,T\right]$. 
\end{defn}
Under this assumption we can construct a solution step by step on
intervals $[0,\delta],[0,2\delta],\ldots$ and so on only supposing,
for instance, that for each $Z\in C^{\alpha}([0,T])$, $t\mapsto F(t,Z)$
is $\gamma$-Hölder continuous with $\alpha+\gamma>1$. 
The theorem below holds even when $U$ and $V$ are generic
Banach spaces.
\begin{thm}
\label{TConstructive} Let $F:\left[0,T\right]\times C^{\alpha}\left(\left[0,T\right],U\right)\rightarrow L$
be a $\delta$-non-anticipating vector field for some $\delta>0$.
Suppose that for each $Z\in C^{\alpha}\left([0,T],U\right)$, $t\mapsto F(t,Z)$
is $\gamma$-Hölder continuous with $\alpha+\gamma>1$. Let $X\in C^{\alpha}\left(\left[0,T\right],V\right)$.
Then for each $y_{0}\in U$ there is a unique solution $Y\in C^{\alpha}([0,T])$
for the equation 
\begin{equation}
Y_{t}=y_{0}+\int_{0}^{t}F\left(u,Y\right)dX_{u},\ t\in\left[0,T\right].\label{eq:EDO_delta}
\end{equation}
\end{thm}
\begin{proof}
We start discussing existence. Without restriction of generality we can suppose that $T=N\delta$
for some integer $N$. We denote by $Y^{0}$ the constant function
$t\mapsto y_{0}\in U$ on $[0,T]$. By recurrence arguments we construct
a sequence $Y^{1},\ldots,Y^{N}$ in $C^{\alpha}([0,T])$ verifying,
for $n=1,\ldots,N$, 
\begin{equation}
Y_{t}^{n}:=y_{0}+\int_{0}^{t\wedge n\delta}F\left(u,Y^{n-1}\right)dX_{u},\ t\in[0,T].\label{eq:Y_def}
\end{equation}
Although we need to define $Y_{t}^{n}$ only for $t\in\left[0,n\delta\right]$,
we have chosen to extend it to the whole interval $t\in\left[0,T\right]$
since this simplifies the formulation of some arguments during this
proof. 

The expression \eqref{eq:Y_def} is well-defined via  Theorem
\ref{thm:Young_integral} with $W_{t}=F\left(t,Y^{n-1}\right)$. Indeed,
suppose that \eqref{eq:Y_def} holds replacing integer $n$ with $n-1$.
Since $Y^{n-1}\in C^{\alpha}([0,T])$, the hypothesis implies
that the path $t\in\left[0,T\right]\mapsto F\left(t,Y^{n-1}\right)$
is $\gamma$-Hölder continuous hence we can define $Y_{t}^{n}:=y_{0}+\int_{0}^{t\wedge n\delta}F\left(u,Y^{n-1}\right)dX_{u}$,
for $t\in\left[0,T\right]$. 

At this point our aim is to prove that $Y:=Y^{N}$ 
solves  equation \eqref{eq:EDO_delta}. For this we will show
that for $n=1,\ldots,N$ 
\begin{equation}
Y_{t}^{n}=y_{0}+\int_{0}^{t\wedge n\delta}F\left(u,Y^{n}\right)dX_{u},\ t\in[0,T].\label{eq:Y_ind}
\end{equation}
Taking into account the construction \eqref{eq:Y_def},
it will be enough to show by recurrence that, for every $n=1,\ldots,N$
\begin{equation}
F(u,Y^{n})=F(u,Y^{n-1}),\ u\in\left[0,n\delta\right].\label{Eqn}
\end{equation}
Suppose for a moment that, for every $n=1,\ldots,N$, 
\begin{equation}
Y_{t}^{n}=Y_{t}^{n-1},\,t\in[0,(n-1)\delta].\label{Eqn1}
\end{equation}
 Then \eqref{Eqn} will follow. Indeed let $u\in[0,n\delta]$, for
some $n$; then $t:=(u-\delta)_{+}$ belongs to $[0,(n-1)\delta]$;
thus from \eqref{Eqn1},  recalling $\mathcal{Y}_{r}\left(x\right):=Y_{r\wedge x}$
it follows 
\begin{equation} \label{Eqn2}
 \mathcal{Y}_{\left(u-\delta\right)_{+}}^{n}=\mathcal{Y}_{\left(u-\delta\right)_{+}}^{n-1},
u\in[0,n\delta].
\end{equation}
So, using \eqref{eq:F_nonant} and taking in account \eqref{Eqn2}, yields
\[
F\left(u,Y^{n}\right)=F\left(u,\mathcal{Y}_{(u-\delta)_{+}}^{n}\right)=F\left(u,\mathcal{Y}_{(u-\delta)_{+}}^{n-1}\right)=F\left(u,Y^{n-1}\right).
\]

It remains to show \eqref{Eqn1}; we will do it by induction on $n$.
The case $n=1$ is an obvious consequence of \eqref{eq:Y_def} evaluated
for $t=0$. We suppose now that \eqref{Eqn1} holds replacing $n$
with $n-1$.




Let $t\in[0,(n-1)\delta]$ for an integer $n$ with $n\ge2$. By \eqref{eq:F_nonant},
the construction \eqref{eq:Y_def}, the induction hypothesis related
to \eqref{Eqn1} and the recurrence \eqref{Eqn} with $n-1$ replacing $n$, we obtain 
\begin{eqnarray*}
Y_{t} & = & y_{0}+\int_{0}^{t\wedge n\delta}F\left(u,Y^{n-1}\right)dX_{u}\\
 & = & y_{0}+\int_{0}^{t\wedge(n-1)\delta}F\left(u,Y^{n-1}\right)dX_{u}\\
 & = & y_{0}+\int_{0}^{t\wedge(n-1)\delta}F\left(u,Y^{n-2}\right)dX_{u}\\
 & = & Y_{t}^{n-1}.
\end{eqnarray*}
This concludes the proof of the induction step in \eqref{Eqn1}. The
existence part of the theorem is finally established.

Uniqueness follows easily by an obvious induction argument.


\end{proof}

{\bf ACKNOWLEDGEMENTS.} 
The first named author was supported by the (Brazilian) National Council 
for Scientific and Technological Development (CNPq), under 
 grant Nr. 232925/2014-3.
 Part of the work was done during a visit of the second author to
the University of Bielefeld, SFB 701. He is grateful to Prof. Michael
R\"ockner for the kind invitation.

 \bibliographystyle{amsplain}
\addcontentsline{toc}{section}{\refname}\bibliography{rafael}

\def\polhk#1{\setbox0=\hbox{#1}{\ooalign{\hidewidth
  \lower1.5ex\hbox{`}\hidewidth\crcr\unhbox0}}}
  \def\polhk#1{\setbox0=\hbox{#1}{\ooalign{\hidewidth
  \lower1.5ex\hbox{`}\hidewidth\crcr\unhbox0}}} \def\cprime{$'$}
  \def\polhk#1{\setbox0=\hbox{#1}{\ooalign{\hidewidth
  \lower1.5ex\hbox{`}\hidewidth\crcr\unhbox0}}}
  \def\polhk#1{\setbox0=\hbox{#1}{\ooalign{\hidewidth
  \lower1.5ex\hbox{`}\hidewidth\crcr\unhbox0}}}
  \def\polhk#1{\setbox0=\hbox{#1}{\ooalign{\hidewidth
  \lower1.5ex\hbox{`}\hidewidth\crcr\unhbox0}}} \def\cprime{$'$}
  \def\cprime{$'$}
\providecommand{\bysame}{\leavevmode\hbox to3em{\hrulefill}\thinspace}
\providecommand{\MR}{\relax\ifhmode\unskip\space\fi MR }
\providecommand{\MRhref}[2]{%
  \href{http://www.ams.org/mathscinet-getitem?mr=#1}{#2}
}
\providecommand{\href}[2]{#2}
\begin{thebibliography}{10}

\bibitem{schauderbonsall1962lectures}
F.~F. Bonsall and K.B. Vedak, \emph{Lectures on some fixed point theorems of
  functional analysis}, no.~26, Tata Institute of Fundamental Research Bombay,
  1962.

\bibitem{cho}
A.~Chojnowska-Michalik, \emph{Representation theorem for general stochastic
  delay equations}, Bull. Acad. Polon. Sci. S\'er. Sci. Math. Astronom. Phys.
  \textbf{26} (1978), no.~7, 635--642.

\bibitem{cosso_russo15a}
A.~Cosso and F.~Russo, \emph{\emph{Functional {I}t\^o versus {B}anach space
  stochastic calculus and strict solutions of semilinear path-dependent
  equations}}, \emph{Preprint} arXiv:1505.02926, 2015.

\bibitem{delfour}
M.~C. Delfour and S.~K. Mitter, \emph{Controllability, observability and
  optimal feedback control of affine hereditary differential systems}, SIAM J.
  Control \textbf{10} (1972), 298--328. \MR{0309587}

\bibitem{DGR}
C.~Di~Girolami and F.~Russo, \emph{Infinite dimensional stochastic calculus via
  regularization and applications}, Preprint \textup{HAL-INRIA, inria-00473947
  version 1} (2010), no.~Unpublished.

\bibitem{DGRnote}
\bysame, \emph{Clark-{O}cone type formula for non-semimartingales with finite
  quadratic variation}, C. R. Math. Acad. Sci. Paris \textbf{349} (2011),
  no.~3-4, 209--214.

\bibitem{dupire2009functional}
B.~Dupire, \emph{Functional it{\^o} calculus}, Bloomberg Portfolio Research
  Paper (2009), no.~2009-04.

\bibitem{flandoli_zanco13}
F.~Flandoli and G.~Zanco, \emph{\emph{An infinite-dimensional approach to
  path-dependent Kolmogorov's equations}}, To appear in \emph{Ann. Probab.},
  2013.

\bibitem{HairerBook}
P.~K. Friz and M.~Hairer, \emph{A course on rough paths}, Springer, 2014.

\bibitem{friz2010multidimensional}
P.~K. Friz and N.~B. Victoir, \emph{Multidimensional stochastic processes as
  rough paths: theory and applications}, Cambridge University Press Cambridge,
  2010.

\bibitem{gubinelli2004controlling}
M.~Gubinelli, \emph{Controlling rough paths}, Journal of Functional Analysis
  \textbf{216} (2004), no.~1, 86--140.

\bibitem{gubinelli2006young}
M.~Gubinelli, A.~Lejay, and S.~Tindel, \emph{Young integrals and spdes},
  Potential Analysis \textbf{25} (2006), no.~4, 307--326.

\bibitem{leao_ohashi_simas14}
D.~Le\~ao, A.~Ohashi, and A.~B. Simas, \emph{\emph{Weak functional {I}t\^o
  calculus and applications}}, \emph{Preprint} arXiv:1408.1423v2, 2014.

\bibitem{lejay:inria-00402397}
A.~Lejay, \emph{{Controlled differential equations as Young integrals: a simple
  approach}}, Journal of Differential Equations \textbf{249} (2010), 1777--1798
  (Anglais).

\bibitem{lyons-young}
T.~Lyons, \emph{Differential equations driven by rough signals. {I}. {A}n
  extension of an inequality of {L}. {C}. {Y}oung}, Math. Res. Lett. \textbf{1}
  (1994), no.~4, 451--464. \MR{1302388}

\bibitem{maslowski}
B.~Maslowski and D.~Nualart, \emph{Evolution equations driven by a fractional
  {B}rownian motion}, J. Funct. Anal. \textbf{202} (2003), no.~1, 277--305.
  \MR{1994773}

\bibitem{mohammed}
S.~E.~A. Mohammed, \emph{Stochastic functional differential equations},
  Research Notes in Mathematics, vol.~99, Pitman (Advanced Publishing Program),
  Boston, MA, 1984.

\bibitem{rw}
L.~C.~G. Rogers and D.~Williams, \emph{Diffusions, {M}arkov processes, and
  martingales. {V}ol. 2}, Cambridge Mathematical Library, Cambridge University
  Press, Cambridge, 2000, It\^{o} calculus, Reprint of the second (1994)
  edition.

\bibitem{young1936inequality}
L.~C. Young, \emph{An inequality of the {H}{\"o}lder type, connected with
  stieltjes integration}, Acta Mathematica \textbf{67} (1936), no.~1, 251--282.

\end{thebibliography}

\end{document}